\documentclass[12pt]{article}
\usepackage[letterpaper,margin=1in]{geometry}
\usepackage{mathrsfs,dsfont}

\usepackage{setspace}
\usepackage{amsmath,amsthm,amssymb}
\usepackage{mathptmx}
\usepackage{hyperref}

\font\tencmmib=cmmib10 \skewchar\tencmmib '60
\newfam\cmmibfam
\textfont\cmmibfam=\tencmmib

\def\lessim{\ \lower4pt\hbox{$
		\buildrel{\displaystyle <}\over\sim$}\ }
\def\gessim{\ \lower4pt\hbox{$\buildrel{\displaystyle >}
		\over\sim$}\ }

\def\eps{{\varepsilon}}

\def\qed{\hfill\hbox{\rlap{$\sqcap$}$\sqcup$}}

\newcommand{\la}{\langle}
\newcommand{\ra}{\rangle}
\renewcommand{\geq}{\geqslant}
\renewcommand{\leq}{\leqslant}

\renewcommand{\le}{\leqslant}
\newcommand{\e}{\mathbb{E}}
\newcommand{\p}{\mathbb{P}}
\newcommand{\Reals}{\mathbb{R}}

\newtheorem{lemma}{\bf Lemma}

\newtheorem{theorem}{\bf Theorem}

\newtheorem{proposition}{\bf Proposition}

\newenvironment{Proof of lemma}{\noindent{\bf Proof of Lemma}}{\hfill$\Box$\newline}
\newenvironment{Proof of theorem}{\noindent{\it Proof of Theorem}}{\hfill\scriptsize{$\Box$}\newline}
\newenvironment{Proof of theorems}{\noindent{\bf Proof of Theorems}}{\hfill$\Box$\newline}
\newenvironment{Proof of proposition}{\noindent{\bf Proof of Proposition}}{\hfill$\Box$\newline}
\newenvironment{Proof of propositions}{\noindent{\bf Proof of Propositions}}{\hfill$\Box$\newline}
\newenvironment{Proof of exercise}{\noindent{\it Proof of Exercise:}}{\hfill$\Box$}

\title{Fluctuations of the free energy in the mixed $p$-spin models with external field}

\author{Wei-Kuo Chen\thanks{School of Mathematics, University of Minnesota. Email: wkchen@umn.edu. Research supported by NSF grant DMS-1513605, 
		NSF-Simon Travel Grant and Hong Kong Research Grants Council GRF-14302515.}
		\and
		Partha Dey\thanks{Department of Mathematics, University of Illionis at Urbana-Champaign. Email: psdey@illinois.edu.}
	\and
	Dmitry Panchenko\thanks{Department of Mathematics, University of Toronto. Email: panchenk@math.toronto.edu. Partially supported by NSERC grant.}
}

\date{}

\begin{document}
	\maketitle
\begin{abstract}
We show that the free energy in the mixed $p$-spin models of spin glasses does not superconcentrate in the presence of external field, which means that its variance is of the order suggested by the Poincar\'e inequality. This complements the result of Chatterjee who showed that the free energy superconcentrates when there is no external field. For models without odd $p$-spin interactions for $p\geq 3$, we prove the central limit theorem for the free energy at any temperature and give an explicit formula for the limiting variance. Although we only deal with the case of Ising spins, all our results can be extended to the spherical models as well.
\end{abstract}

	\section{Introduction}
	
	In~\cite{Chatt08, Chatt09, Chatt14} Chatterjee developed a theory that linked various phenomena -- superconcentration, chaos, and multiple valleys -- for general Gaussian fields and gave a number of examples of application (see also \cite{Chatt13, Ding13}). One of the examples related to spin glass models showed that the free energy in the Sherrington-Kirkpatrick model \cite{SK} without external field superconcentrates at any temperature, i.e., it has variance of a smaller order than the usual one suggested by the Poincar\'e inequality (see~\cite[Theorem~$1.5$]{Chatt09}, or~\cite[Sections~$6.3, 10.2$]{Chatt14}). Before that, this was known (in a stronger form) only at high temperature (see~\cite{ALR} or~\cite[Section 11.4]{Tal11}). Chatterjee's techniques can also be applied to show superconcentration in a more general class of mixed $p$-spin models without external field. In this paper we complement these results by showing that there is no superconcentration in the presence of external field and moreover, in the case of mixed even $p$-spin models with external field, we obtain a Gaussian central limit theorem for the free energy at any temperature.

	First, let us recall the definition of mixed $p$-spin models. Let $(\beta_p)_{p\geq 1}$ be a sequence of non-negative real numbers decreasing fast enough, for example, satisfying $\sum_{p\geq 1}2^p\beta_p^2<\infty$. The mixed $p$-spin Hamiltonian is defined as a linear combination
\begin{equation}
	H_N(\sigma):=\sum_{p\geq 2} \beta_p H_{N,p}(\sigma)+\sum_{i\leq N}(h+ \beta_1 g_i)\sigma_i, \quad \sigma\in \{\pm1\}^N,
\label{Ham}
\end{equation}
where the $p$-th term in the first sum
\begin{equation}
H_{N,p}(\sigma) := \frac{1}{N^{(p-1)/2}}\sum_{1\leq i_1,\ldots,i_{p}\leq N}g_{i_1,\ldots,i_{p}}\sigma_{i_1}\cdots\sigma_{i_{p}}
\end{equation}
is called the pure $p$-spin Hamiltonian, $(g_{i_1,\ldots,i_{p}})$ are i.i.d.~standard Gaussian for all $p\geq 1$ and  $(i_1,\ldots,i_{p})\in \{1,2,\ldots,N\}^{p}$, and $h\in \Reals$. The last sum in (\ref{Ham}) is called the external field term, where we separated the parameters into a non-random part $h$ and symmetric Gaussian part $\beta_1 g_i$. Consider the following quantities
\begin{equation}
Z_N := \sum_{\sigma} \exp H_N(\sigma),\quad
f_N := \log Z_N, \quad
F_N := \frac{f_N}{N}
\end{equation}
-- the partition function, random unscaled and scaled free energy. Our first result is the following.
	\begin{theorem}\label{thm1}
		If the external field term is present, i.e.,~$h^2 + \beta_1^2\neq 0,$ then		\begin{align}\label{thm:eq1}
		{cN}\leq \mathrm{Var}(f_N)\leq {CN}
		\end{align}
for some constants $ C>c>0$ independent of $N$.
	\end{theorem}
The proof is based on a version of Chatterjee's representation for the variance (see Lemma \ref{lem2} below)  and some consequences of the validity of the Parisi formula for the free energy. 

In the case when all odd $p\geq 3$ spin terms in the mixed $p$-spin model vanish, we have additional tools available from the theory of spin glasses, which will allow us to strengthen Theorem \ref{thm1} and prove the central limit  theorem for the free energy. 

The description of the variance of the limiting Gaussian distribution is explicit but quite complicated, and we need to recall several results and definitions first. Consider the following function
\begin{equation}
\xi(x):= \sum_{p\geq 1}\beta_p^2x^{p}.
\label{eta}
\end{equation}
The Parisi formula for the free energy \cite{Parisi79, Parisi80}, which was proved for mixed even $p$-spin spin models by Talagrand in \cite{Tal06} and for general mixed $p$-spin models in \cite{P13, P14} (see~\cite[Chapter 3]{SKmodel}), states that 
	\begin{equation}
	\lim_{N\rightarrow\infty}\e F_N=\min_{\mu}\Bigl(\log 2+ \e\Phi_{\mu}(0,h+\beta_1 g_1)-\frac{1}{2}\int_0^1\! \xi''(q)q\mu([0,q])\, dq\Bigr),
	\label{Parisif}
	\end{equation}
	where the minimum is taken over all probability measures $\mu$ on $[0,1]$ and $\Phi_{\mu}(q,x)$ for $q\in [0,1]$ and $x\in\Reals$ is the solution of the Parisi equation
\begin{equation}
	\partial_q\Phi_{\mu}=-\frac{\xi''(q)}{2}\Bigl(\partial_{xx}\Phi_{\mu}+\mu([0,q]) \bigl(\partial_x\Phi_{\mu}\bigr)^2\Bigr)
	\label{ParisiEq}
\end{equation}
with the boundary condition $\Phi_{\mu}(1,x)=\log \cosh x.$ It was proved in \cite{AChen14} (see also \cite{Jagannath}) that the Parisi variational formula has unique minimizer, which will be denoted by $\mu_P$. When $h^2 + \beta_1^2\neq 0$,~\cite[Theorem 14.12.1]{Tal11} proves that the support of the Parisi measure is separated from zero, i.e.,~$d:=\min\mbox{supp}\,\mu_P>0.$ Furthermore, by~\cite[Proposition 1]{Chen12} we have that
\begin{align}\label{eqd}
d&=\e \bigl(\partial_x\Phi_{\mu_P}(d,h+\chi)\bigr)^2,
\end{align}
where $\chi$ is a centered Gaussian random variable with variance $\xi'(d).$ Next, for any fixed $t\in (0,1)$, we consider the function
\begin{align}
\label{defphi}
\varphi_t(s):=\e \partial_x\Phi_{\mu_P}\bigl(d,h+\chi_t^1(s)\bigr)\partial_x\Phi_{\mu_P}\bigl(d, h+\chi_t^2(s)\bigr)
\end{align}
of $s\in [0,d]$, where the pair $(\chi_t^1(s),\chi_t^2(s))$ is centered Gaussian with the covariance given by
\begin{align*}
\e \chi_t^1(s)^2&=\e \chi_t^2(s)^2=\xi'(d),\quad \e \chi_t^1(s)\chi_t^2(s)=t\xi'(s).
\end{align*}
In the setting of mixed even $p$-spin models, it was shown in~\cite[Proposition~$1$]{Chen12} that $\varphi_t(s)$ has a unique fixed point in $[0,d]$, which we denote by $u_t$. Clearly, $u_{1}=d$. Finally, we define
\begin{align}
\label{eq4}
\nu :=\int_0^1\! \xi(u_t)\, dt.
\end{align}
This quantity is precisely the limiting scaled variance in our central limit theorem. Before stating the central limit theorem  let us recall that the \emph{total variation distance} between two r.v.s $X$ and $Y$ is defined as 
\[
d_{\mathrm{TV}}(X,Y):=\sup_{A}|\p(X\in A) - \p(Y\in A)|.
\]
\begin{theorem}\label{thm2}
Assume that $\beta_p=0$ whenever $p\geq 3$ is odd and the external field is present, i.e.,~$h^2+\beta_1^2\neq 0$. For $\nu$ defined in (\ref{eq4}), we have
\begin{equation}
\lim_{N\rightarrow\infty}d_{\mathrm{TV}}\Bigl(\frac{f_N-\e f_N}{\sqrt{\nu N}},g\Bigr)=0,
\label{CLTeq}
\end{equation}	
where $d_{\mathrm{TV}}$ is the total variation distance and $g$ is standard Gaussian.
\end{theorem}

Let us explain right away why the condition $h^2+\beta_1^2 \not= 0$ implies that $u_t>0$ for all $t>0$ and, thus, $\nu> 0.$ By the definition of $u_t$ as the fixed point, it is enough to check that $\varphi_t(0)> 0$. For $s=0$, the covariance of $(\chi_t^1(0),\chi_t^2(0))$ above can be rewritten as
\begin{align*}
\e \chi_t^1(0)^2&=\e \chi_t^2(0)^2=\beta_1^2 + a^2\,\,\mbox{ and }\,\,\e \chi_t^1(0)\chi_t^2(0)=t \beta_1^2,
\end{align*}
where $a^2 = \xi'(d) - \beta_1^2 = \sum_{p\geq 2}p\beta_p^2 d^{p-1}.$ Therefore, we can define
$$
\chi_t^1(0) =  \beta_1 \sqrt{t} z +  + \sqrt{\beta_1^2(1-t) + a^2}\, g_1,\,\,
\chi_t^2(0) = \beta_1 \sqrt{t} z + \sqrt{\beta_1^2(1-t) + a^2}\, g_2
$$
for independent standard Gaussian random variables $z,g_1,g_2$ and, by (\ref{defphi}) we have
\begin{align*}
\varphi_t(0)&=\e \bigl(\e_1 \partial_x\Phi_{\mu_P}\bigl(d,h+\beta_1 \sqrt{t} z + \sqrt{\beta_1^2(1-t) + a^2}\, g_1\bigr)\bigr)^2,
\end{align*}
where $\e_1$ is the expectation in $g_1.$ It is well known that for any $q\in [0,1]$, $\Phi_{\mu}(q,\,\cdot\,)$ is symmetric and strictly convex (see e.g.~\cite[Lemma 14.7.16]{Tal11} and  \cite[Proposition~$2$]{AChen13}). Therefore, $\partial_x\Phi_{\mu_P}(d,\,\cdot\,)$ is odd and strictly increasing and the expectation $\e_1$ inside the square is not zero when $h+\beta_1\sqrt{t}z \not =0$, which happens with probability one when $h^2 + \beta_1^2 \not =0$ and $t>0$. This shows that $u_t> 0$ and $\nu>0.$

To prove Theorem \ref{thm2}, we adapt Stein's method to control the total variation distance for the free energy through a covariance formula. The crucial part of the argument is played by a recent result on the disorder chaos in the mixed even $p$-spin model with external field obtained in \cite{Chen12}, which allows us to gain control of the total variation distance and determine the exact value of $\nu$ in \eqref{eq4}. 

We remark that the approaches in Theorems \ref{thm1} and \ref{thm2} can be applied to the spherical version of the model as well, when $\sigma\in S^{N-1}$, with the spin glass computations for the Ising spins replaced by the corresponding results in \cite{Tal-sphere, Chen-sphere, Chen15} for the spherical case. In the spherical Sherrington-Kirkpatrick model without external field, it was shown by Baik and Lee in \cite{Baik15} that the free energy superconcentrates at high temperature with $\mbox{Var}(f_N)=\Theta(1)$, while at low temperature it superconcentrates with $\mathrm{Var}(f_N) = \Theta(N^{2/3})$ and the fluctuations around the limiting value of the free energy are given by the GOE Tracy-Widom distribution. In contrast, in the presence of external field our results show that the order of  fluctuations of the free energy is the same at any temperature and, for even $p$-spin models, the classical Gaussian central limit theorem holds. Of course, it would be of great interest to extend Theorem \ref{thm2} to include the case of odd spin interactions and, if possible, obtain the rate of convergence in (\ref{CLTeq}).

\section{Proof of Theorem \ref{thm1}}
	
The upper bound is a standard application of the Poincar\'e inequality, so only the lower bound requires proof. We will start with a version of Chatterjee's representation for the variance and some of its properties, Lemmas 3.4 and 3.5 in \cite{Chatt08}.

	Let $z, z^1$ and $z^2$ be independent standard Gaussian vectors on $\Reals^n$ and, for $t\in [0,1]$, define
	\begin{equation}
	z^1(t)=\sqrt{t}z+\sqrt{1-t}z_1\,\,\mbox{and}\,\,z^2(t)=\sqrt{t}z+\sqrt{1-t}z_2.
	\end{equation}
	For a function $f:\Reals^n\to\Reals$ such that $\e f(z)^2<\infty$, let
	\begin{align}\label{lem1:eq1}
	\varphi(t)=\e f\bigl(z^1(t)\bigr)f\bigl(z^2(t)\bigr).
	\end{align}
	This quantity is, clearly, nonnegative since, by symmetry,
	\begin{equation}
	\varphi(t) = \e\bigl(\e_1 f\bigl(z^1(t) \bigr)\bigr)^2\geq 0,
	\label{phipos}
	\end{equation}
	where $\e_1$ is the expectation with respect to $z_1$. If all partial derivatives $\partial_i f$ of $f$ are of moderate growth then, taking derivative and using Gaussian integration by parts,
	\begin{align}\label{lem1:eq2}
	\varphi'(t)&=\sum_{j=1}^n\e\partial_{j}f\bigl(z^1(t)\bigr)\partial_{j}f\bigl(z^2(t)\bigr)\geq 0,
	\end{align}
	which is nonnegative because each term is of the form (\ref{lem1:eq1}). This means that $\varphi(t)$ is non-decreasing. Actually, this fact also holds for any $f$ such that $\e f(z)^2<\infty$ (see Lemma 3.5 in \cite{Chatt08}), but here we will deal only with nice smooth functions. Notice that, by induction, $\varphi^{(k)}(t)\geq 0$ as long as partial derivatives of order $k$ are of moderate growth. This was observed in Lemma 3.3 in \cite{Chatt09} with some important consequences.

Let us consider the Gibbs measure corresponding to the Hamiltonian (\ref{Ham}),
\begin{equation}
G_N(\sigma) = \frac{\exp H_N(\sigma)}{Z_N},
\end{equation}
and denote by $\la\,\cdot\,\ra$ the average with respect to $G_N^{\otimes \infty}$. Recall the Hamiltonian (\ref{Ham}) and let
\begin{equation}
	Y_N(\sigma)=\sum_{p\geq 2} \beta_p H_{N,p}(\sigma)+ \beta_1 \sum_{i\leq N} g_i \sigma_i
\label{HamY}
\end{equation}
be its random Gaussian part, excluding non-random external field $h$. Consider two independent copies $Y_N^1$ and $Y_N^2$ of $Y_N$ and, for $t\in[0,1]$, define two correlated copies of the Hamiltonian (\ref{Ham}),
	\begin{align}
	H_{t}^1(\sigma)&=\sqrt{t}Y_N(\sigma)+\sqrt{1-t}Y_N^1(\sigma)+h\sum_{i=1}^N\sigma_i,
	\nonumber
	\\
	H_{t}^2(\tau)&=\sqrt{t} Y_N(\tau)+\sqrt{1-t} Y_N^2(\tau)+h\sum_{i=1}^N\tau_i.
	\label{Hamcor}
	\end{align}
Let $G_{t}^1(\sigma)$ and $G_{t}^2(\tau)$ be the Gibbs measures corresponding to these Hamiltonians. Denote $\la\,\cdot\,\ra_{t}$ the Gibbs average with respect to the product measure $G_{t}^1\times G_{t}^2.$ Let us denote by
\begin{equation}
R(\sigma,\tau) = \frac{1}{N}\sum_{i=1}^N \sigma_i \tau_i
\end{equation}
the overlap between configurations $\sigma,\tau\in\{-1,+1\}^N$.
	\begin{lemma}\label{lem2}
		The following representation holds,
		\begin{align}
		\begin{split}\label{lem2:eq1}
		\mathrm{Var}(f_N)&=N\int_0^1\! \e\bigl\la\xi\bigl(R(\sigma,\tau)\bigr) \bigr\ra_{t}\, dt,
		\end{split}
		\end{align}	
		and the integrand $\e\la\xi(R(\sigma,\tau))\ra_{t}$ is nonnegative and non-decreasing in $t$.
	\end{lemma}
	
	\begin{proof}
		Let us consider the following function of $t\in [0,1]$,
		\begin{align*}
		\varphi(t)&= \e \Bigl(\log \sum_{\sigma}\exp H_{t}^1(\sigma)\Bigr)\Bigl(\log \sum_{\tau}\exp H_{t}^2(\tau)\Bigr).
		\end{align*}
		It is easy to see from the definition (\ref{Hamcor}) that
		$$\varphi(1)=\e f_N^2 \,\,\mbox{ and }\,\, \varphi(0)=(\e f_N)^2.$$
		Using Gaussian integration by parts, one can check that
		$$
		\varphi'(t)=N\e\bigl\la\xi\bigl(R(\sigma,\tau)\bigr)\bigr\ra_{t}
		$$ 
		and \eqref{lem2:eq1} follows. The reason this derivative is nonnegative and non-decreasing is because, for any $p\geq 1$,
\begin{align*}
& N \e\bigl\la \bigl(R(\sigma,\tau)\bigr)^p\bigr\ra_{t} 
=
\sum_{i_1,\ldots, i_p} \e\bigl\la \sigma_{i_1}\tau_{i_1}\cdots \sigma_{i_p}\tau_{i_p}\bigr\ra_{t} 
=
\sum_{i_1,\ldots, i_p} \e\bigl\la \sigma_{i_1}\cdots \sigma_{i_p}\bigr\ra_{1} \bigl\la \tau_{i_1}\cdots \tau_{i_p}\bigr\ra_{2}
\end{align*}
(where $\la\,\cdot\,\ra_1$ and $\la\,\cdot\,\ra_2$ denote the Gibbs averages with respect to $G_{t}^1(\sigma)$ and $G_{t}^2(\tau)$) and each term is of the form (\ref{lem1:eq1}).  Even though these are now functions of possibly infinitely many i.i.d. Gaussians $g_{i_1,\ldots, i_p}$, we can approximate by functions of finitely many Gaussians by truncating the series in (\ref{Ham}) at large finite $p$.
	\end{proof}

\medskip
Lemma~\ref{lem2} implies that, for any $t<1$,
\begin{equation}
\mathrm{Var}(f_N) \geq N(1-t)\e\bigl\la\xi\bigl(R(\sigma,\tau)\bigr) \bigr\ra_{t}.
\label{eq22}
\end{equation}
One can now use the disorder chaos results in \cite{Chen12} (or \cite{Chen15} for the spherical model), to show that $R(\sigma,\tau)$ concentrates on a constant value $u_t$ under $\la\,\cdot\,\ra_t$ and check that $u_t> 0$ when $h^2+\beta_1^2 \not =0.$ The results of \cite{Chen12} only apply to even $p$-spin models and will be used in the next section but here we will instead give a simpler approach which will also apply to general mixed $p$-spin models that include odd $p$-spin interactions.

First, we will fix any $t\in(0,1)$ in \eqref{eq22} and will freeze the interpolation of the external field terms in \eqref{Hamcor} at time $t$, while at the same time continuing the interpolation of $p$-spin interaction terms using another parameter $s\in [0,t].$ More precisely, we will break the Hamiltonian in (\ref{HamY}) into two components,
\begin{equation}
X_N(\sigma)=\sum_{p\geq 2} \beta_p H_{N,p}(\sigma),\,\,
Z_N(\sigma) = \beta_1 \sum_{i\leq N} g_i \sigma_i,
\end{equation}
consider their independent copies $X_N^1, Z_N^1$ and $X_N^2, Z_N^2$ and, for $s\in[0,t]$, define
	\begin{align}
\hat{H}_{s}^1(\sigma)&=
\sqrt{s}X_N(\sigma)+\sqrt{1-s}X_N^1(\sigma) +
\sqrt{t}Z_N(\sigma)+\sqrt{1-t}Z_N^1(\sigma)+h\sum_{i=1}^N\sigma_i,
	\nonumber
	\\
\hat{H}_{s}^2(\tau)&=\sqrt{s} X_N(\tau)+\sqrt{1-s} X_N^2(\tau)
+ \sqrt{t} Z_N(\tau)+\sqrt{1-t} Z_N^2(\tau)+h\sum_{i=1}^N\tau_i.
	\label{Hamcor2}
	\end{align}
Let $G_{t,s}^1(\sigma)$ and $G_{t,s}^2(\tau)$ be the Gibbs measures corresponding to these Hamiltonians and let $\la\,\cdot\,\ra_{t,s}$ be the Gibbs average with respect to the product measure $G_{t,s}^1\times G_{t,s}^2.$ Clearly, for $s=t$ this coincides with the previous definition, $\hat{H}_{t}^\ell= {H}_{t}^\ell$ for $\ell=1,2$ in distribution, and 
$$
\e\bigl\la\xi\bigl(R(\sigma,\tau)\bigr) \bigr\ra_{t}
=
\e\bigl\la\xi\bigl(R(\sigma,\tau)\bigr) \bigr\ra_{t,t}.
$$ 
On the other hand, the function $s\to \e\la\xi(R(\sigma,\tau)) \ra_{t,s}$ is still nonnegative and non-decreasing, because the calculations and symmetry considerations in the equations \eqref{lem1:eq1}--\eqref{lem1:eq2} apply to each term $\e\la \sigma_{i_1}\tau_{i_1}\cdots \sigma_{i_p}\tau_{i_p}\ra_{t,s}$. Together with \eqref{eq22}, this yields
\begin{equation}
\mathrm{Var}(f_N) \geq N(1-t)\e\bigl\la\xi\bigl(R(\sigma,\tau)\bigr) \bigr\ra_{t,0}.
\label{eq23}
\end{equation}
If we denote, for $\ell=1,2$ and $i\leq N$,
\begin{equation}
h_i^\ell=h+\beta_1(\sqrt{t}g_i + \sqrt{1-t}g_i^\ell),
\label{hscort}
\end{equation}
then we can rewrite
$$
\hat{H}_0^1(\sigma)=X_N^1(\sigma)+\sum_{i=1}^Nh_i^1\sigma_i
\,\,\mbox{ and }\,\,
\hat{H}_0^2(\tau)=X_N^2(\sigma)+\sum_{i=1}^Nh_i^2\tau_i.
$$
We will now show that, under $\la\,\cdot\,\ra_{t,0}$, the overlap $R(\sigma,\tau)$ concentrates near some constant $u$, which will be strictly positive when $h^2 + \beta_1^2 \not = 0$, finishing the proof of Theorem \ref{thm1}.

Let us recall the definition of the Parisi measure $\mu_P$ and the function $\Phi_{\mu_P}$ in (\ref{Parisif}) and (\ref{ParisiEq}), recall the notation in (\ref{hscort}) and define
\begin{equation}
u=\e\partial_x\Phi_{\mu_P}(0,h_1^1)\partial_x\Phi_{\mu_P}(0,h_1^2).
\label{defu}
\end{equation}
The following holds.
\begin{proposition}\label{prop}
For any $\varepsilon>0,$ there exists $K>0$ independent of $N$ such that
\begin{align*}
\e\bigl\la I\bigl(|R(\sigma,\tau)-u|\geq \varepsilon \bigr) \bigr\ra_{t,0}\leq Ke^{-N/K}.
\end{align*}
\end{proposition}
To see that this finishes the proof of Theorem \ref{thm1}, we rewrite
$$
u = \e\Bigl(\e_1\partial_x\Phi_{\mu_P}\bigl(0,h+\beta_1\sqrt{t}g_1 + \beta_1\sqrt{1-t}g_1^1)\bigr)\Bigr)^2,
$$
where $\e_1$ is the expectation with respect to $g_1^1$. As we mentioned in the introduction, it is well known that for any $q\in [0,1]$, $\Phi_{\mu}(q,\,\cdot\,)$ is symmetric and strictly convex. Therefore, $\partial_x\Phi_{\mu_P}(0,\,\cdot\,)$ is odd and strictly increasing and the expectation $\e_1$ inside the square is not zero whenever $h+\beta_1\sqrt{t}g_1 \not =0$, which happens with probability one when $h^2 + \beta_1^2 \not =0$ and $t>0$.

\begin{proof}[Proof of Proposition \ref{prop}.]
For any $\eps>0$, we define
\begin{align*}
\hat{F}_{N,\varepsilon}^+&=\frac{1}{N}\log \sum_{R(\sigma,\tau)>\varepsilon+u}\exp\bigl(\hat{H}_0^1(\sigma)+\hat{H}_0^2(\tau) \bigr),\\
\hat{F}_{N,\varepsilon}^-&=\frac{1}{N}\log \sum_{R(\sigma,\tau)< -\varepsilon+u}\exp\bigl(\hat{H}_0^1(\sigma)+\hat{H}_0^2(\tau)\bigr).
\end{align*}
Note that, for any $\lambda\geq 0,$
		\begin{align*}
		\e \hat{F}_{N,\varepsilon}^+&\leq \frac{1}{N}\e\log \sum_{R(\sigma,\tau)>\varepsilon+u}\exp \bigl(\hat{H}_0^1(\sigma)+\hat{H}_0^2(\tau)+\lambda N(R(\sigma,\tau)-(\varepsilon+u))\bigr)\\
		&\leq \frac{1}{N}\e\log \sum_{\sigma,\tau}\exp\bigl(\hat{H}_0^1(\sigma)+\hat{H}_0^2(\tau)+\lambda NR(\sigma,\tau)\bigr)-\lambda(\varepsilon+u)
		\end{align*}
		and
		\begin{align*}
		\e \hat{F}_{N,\varepsilon}^-&\leq \frac{1}{N}\e\log \sum_{R(\sigma,\tau)<-\varepsilon+u}\exp \bigl(\hat{H}_0^1(\sigma)+\hat{H}_0^2(\tau)+\lambda N((-\varepsilon+u)-R(\sigma,\tau))\bigr)\\
		&\leq \frac{1}{N}\e\log \sum_{\sigma,\tau}\exp\bigl(\hat{H}_0^1(\sigma)+\hat{H}_0^2(\tau)-\lambda NR(\sigma,\tau)\bigr)-\lambda(\varepsilon-u).
		\end{align*}
If we denote the first terms on the right hand side by
$$
\hat{F}_{N}^{\pm}(\lambda) = 
\frac{1}{N}\e\log \sum_{\sigma,\tau}\exp\Bigl(X_N^1(\sigma)+X_N^2(\tau)
+\sum_{i=1}^N h_i^1\sigma_i + \sum_{i=1}^N h_i^2\tau_i 
\pm\lambda \sum_{i=1}^N \sigma_i\tau_i\Bigr)
$$
then we have shown that, for any $\lambda\geq 0,$ 
\begin{align}
\label{eq-1}
\e \hat{F}_{N,\varepsilon}^{\pm} \leq \e \hat{F}_{N}^{\pm}(\lambda) - \lambda(\eps\pm u).
\end{align}
The key observations now is that, since $X_N^1$ and $X_N^2$ are independent, one can run two independent copies of the Guerra's replica symmetry breaking scheme \cite{G03} with the same order parameter $\mu$ to obtain the following upper bound,
\begin{equation}
\label{eq0}
\e \hat{F}_{N}^{\pm}(\lambda)
\leq 
2\log 2+\e \Psi_\mu(\pm\lambda,0,h_1^1,h_1^2)-\int_0^1\! \xi''(q)q\mu([0,q])\, dq,
\end{equation}
where $\Psi_{\mu}(\lambda,q,x_1,x_2)$ satisfies
\begin{align*}
\partial_q\Psi_{\mu}&=-\frac{\xi''(q)}{2}\Bigl(\partial_{x_1}^2\Psi_{\mu}+\partial_{x_2}^2\Psi_{\mu}+\mu([0,q])\bigl((\partial_{x_1}\Psi_{\mu})^2+(\partial_{x_2}\Psi_{\mu})^2\bigr)\Bigr)
\end{align*}
for $(\lambda,q,x_1,x_2)\in\mathbb{R}\times[0,1]\times\mathbb{R}^3$ with the boundary condition 
$$
\Psi_{\mu}(\lambda,1,x_1,x_2)=\log\bigl(\cosh x_1\cosh x_2\cosh \lambda+\sinh x_1\sinh x_2\sinh\lambda\bigr).
$$
Of course, this boundary condition comes from the identity
$$
\frac{1}{4}\sum_{\eps_1,\eps_2 = \pm 1}
\exp\bigl(\eps_1 x_1 + \eps_2 x_2 + \lambda \eps_1 \eps_2\bigr)
=
\cosh x_1\cosh x_2\cosh \lambda+\sinh x_1\sinh x_2\sinh\lambda.
$$
This type of calculation is completely standard (for analogous computations see e.g. Theorem 15.7.3 in \cite{Tal11}) and the upper bound can be first proved for discrete distributions $\mu$, in which case it can be expressed either via explicit recursive definition as in the original work of Guerra \cite{G03} or via Ruelle probability cascades as, for example, in Chapter 3 in \cite{SKmodel}. For general non-discrete $\mu$ this definition can be extended by approximation, and we give a representation via the above differential equation only for convenience of notations and refer to \cite{Chen150} for details. One important remark is that, as in the case of one system, the original proof of Guerra \cite{G03} will work only for even $p$-spin model. However, as was noticed by Talagrand in \cite{Tal03}, one can obtain the same bound in the general case by introducing a small perturbation of the Hamiltonian and utilizing the Ghirlanda-Guerra identities \cite{Guerra95, Guerra96} (see e.g. Theorem 14.4.4. in \cite{Tal11} or Theorem 3.5 in \cite{SKmodel}). Exactly the same perturbation will work for the above system coupled by the term $\pm\lambda NR(\sigma,\tau)$ to give the upper bound (\ref{eq0}) for the general mixed $p$-spin model.

Finally, we use the bound (\ref{eq0}) as follows. First, it is obvious that
\begin{align*}
\Psi_{\mu}(0,0,x_1,x_2)&=\Phi_{\mu}(0,x_1)+\Phi_{\mu}(0,x_2)
\end{align*}
and, therefore, $\e \Psi_{\mu}(0,0,h_1^1,h_1^2)=2\e\Phi_{\mu}(0,h_1)$ so, for $\mu = \mu_P$, the right hand side of \eqref{eq0} is twice the right hand side of (\ref{Parisif}). By Lemma 4 in \cite{Chen150},
$$
\partial_\lambda\Psi_{\mu}(0,0,x_1,x_2)=\partial_{x_1}\Phi_{\mu}(0,x_1)\partial_{x_2}\Phi_{\mu}(0,x_2)
$$
(also, one can easily check this first for discrete $\mu$ using the representation in terms of Ruelle probability cascades, as in Chapter 3 in \cite{SKmodel}, and then extend  to general $\mu$ by approximation), which implies that
\begin{align*}
\partial_\lambda\Bigl(\e\Psi_{\mu}(\pm\lambda,0,h_1^1,h_1^2)-\lambda(\varepsilon\pm u)\Bigr)\Big|_{\lambda=0}&
=
\pm \e \partial_{x}\Phi_{\mu}(0,h_1^1)\partial_{x}\Phi_{\mu}(0,h_1^2)
-(\varepsilon\pm u).
\end{align*}
Recalling the definition of $u$ in \eqref{defu}, this derivative equals $-\eps$ for $\mu=\mu_P.$ As a result, setting $\mu=\mu_P$ and choosing $\lambda>0$ sufficiently small in (\ref{eq-1}) and (\ref{eq0}), we get strict inequality
\begin{align*}
\limsup_{N\rightarrow\infty}\e\hat{F}_{N,\varepsilon}^{\pm}&<2\lim_{N\rightarrow\infty}\e F_N.
\end{align*}
Applying Gaussian concentration inequalities to $\hat{F}_{N,\varepsilon}^{\pm}$ and $F_N$ finishes the proof.
\end{proof}

\section{Proof of Theorem \ref{thm2}}

Throughout this section, we assume that $\beta_{p}=0$ for all odd $p\geq 3$ and $h^2+\beta_1^2\neq 0.$ Our approach is based on Stein's method (see~\cite{CGS}) of normal approximation, which essentially utilizes the idea that if a random variable $W$ approximately satisfies $\e W\psi(W)\approx\e\psi'(W)$ for a large class of functions $\psi$, then $W$ is approximately standard Gaussian. Here we mention that when $W$ is standard Gaussian, $\e W\psi(W)=\e \psi'(W)$ for all absolutely continuous function $\psi$ for which both expectations are well-defined. More precisely, Stein's lemma~\cite[page 25]{stein}) says that for a standard Gaussian random variable $g$ and any random variable $W$, 
	\begin{align}
	\label{steinlem}
	d_{\mathrm{TV}}(W,g)\le \sup\bigl\{|\e(W\psi(W)-\psi'(W))| : \|\psi'\|_\infty\le 2 \bigr\}.
	\end{align}
Now suppose that, for some $W$, there exists a function $f$ such that
\[
\e W\psi(W)=\e f(W)\psi'(W) 
\]
for all absolutely continuous functions $\psi.$ If $f(W)$ is concentrated at $1$, then we can conclude that $\e W\psi(W)\approx\e\psi'(W)$ and it would follow by Stein's method that the distribution of $W$ is approximately standard Gaussian. This approach has been used in~\cite{cha09} to prove second order Poincar\'e inequalities and is the main ingredient in our proof of Theorem~\ref{thm2}. 

Another crucial ingredient is played by the recent result on disorder chaos in the mixed even $p$-spin model with external field (see \cite[Theorem 1]{Chen12}), which states that, for any $0<t<1$ and any $\varepsilon>0$, there exists some $K>0$ such that
\begin{align}\label{eq2}
\e\bigl\la I\bigl(|R(\sigma,\tau)-u_t|\geq \varepsilon \bigr)\bigr\ra_t\leq Ke^{-N/K},
\end{align}
where the constant $u_t$ was defined in the introduction as the unique fixed point of (\ref{defphi}) on $[0,d]$. A result similar to~\eqref{eq2} is missing for the mixed odd-spin model  to complete the proof of the central limit theorem. We remark that, even though the results in~\cite{Chen12} were stated and proved for the two systems with the same external field, the argument works without any changes in the present setting of correlated external fields.

In order to make a connection between Stein's method and the disorder chaos, we need a generalization of \eqref{lem2:eq1} for the covariance of the functions of Gaussian vectors. Let $y,y_1$ and $y_2$ be independent centered Gaussian vectors on $\Reals^n$ with the covariance matrix $C=(C_{j,j'})$ and, for $0\leq t\leq 1,$ define
\begin{align*}
y^1(t)=\sqrt{t}y+\sqrt{1-t}y_1\,\,\mbox{ and }\,\,y^2(t)=\sqrt{t}y+\sqrt{1-t}y_2.
\end{align*}
Let $A,B:\mathbb{R}^n\rightarrow \Reals$ be absolutely continuous functions such that
$$
\e\|\nabla A(y)\|_2^2<\infty \,\,\mbox{ and }\,\, \e\|\nabla B(y)\|_2^2 <\infty.
$$ Using the Gaussian integration by parts, one can easily check that 
\begin{align}
\label{cov}
\e A(y)B(y)-\e A(y) \e B(y)&=\int_0^1\! \sum_{j,j'=1}^nC_{j,j'}\e\partial_{j}A(y^1(t))\partial_{j'}B(y^2(t))\,dt.
\end{align}	
Recall the definition of $\nu$ in (\ref{eq4}) and define 
$$
W_N=\frac{f_N-\e f_N}{\sqrt{\nu N}}.
$$
Let $g$ be a standard Gaussian r.v.~on $\mathbb{R}$ and $\psi$ be any absolutely continuous function on $\mathbb{R}$ with $\|\psi'\|_\infty\le 2$. We now apply \eqref{cov} with $n=2^N$, $y=Y_N$ defined in (\ref{HamY}) and functions $A=W_N$ and $B=\psi(W_N).$ Recall the definition of correlated copies $H_t^1, H_t^2$ of the Hamiltonian $H_N$ and the Gibbs average $\la\,\cdot\,\ra_t$ in (\ref{Hamcor}). Since $\e W_N=0$, one can check that~\eqref{cov} becomes
\begin{align*}
&\e W_N\psi(W_N)=\frac{1}{\nu}\int_0^1\! \e \psi'(W_{N,t}^1)\bigl\la \xi(R(\sigma,\tau))\bigr\ra_t\,dt,
\end{align*}
where $W_{N,t}^1$ is defined by replacing $H_N$ in $W_N$ by $H_t^1.$ Since $\e\psi'(W_{N,t}^1)=\e\psi'(W_{N})$ for all $t$, the definition of $\nu$ in (\ref{eq4}) implies that
\begin{align}\label{eq3}
\e W_N\psi(W_N)-\e\psi'(W_N)&=\frac{1}{\nu}\int_0^1\! \e\psi'(W_{N,t}^1)\bigl\la \xi(R(\sigma,\tau))-\xi(u_t)\bigr\ra_t\, dt
\end{align}
and, therefore, for any $0<\delta<1/2,$
\begin{align*}
\bigl|\e W_N\psi(W_N)-\e\psi'(W_N)\bigr|&\leq \frac{\|\psi'\|_\infty}{\nu}\int_\delta^{1-\delta}\! \e\bigl\la |\xi(R(\sigma,\tau))-\xi(u_t)| \bigr\ra_t\,dt+\frac{4\delta\xi(1)\|\psi'\|_\infty}{\nu}.
\end{align*}
By Stein's lemma, \eqref{steinlem}, we showed that
\begin{align*}
d_{\mathrm{TV}}(W_{N},g)&\leq \frac{2}{\nu}\int_\delta^{1-\delta}\! \e\bigl\la |\xi(R(\sigma,\tau))-\xi(u_t)|\bigr\ra_t\,dt+\frac{8\delta\xi(1)}{\nu}.
\end{align*}
Using the disorder chaos result in \eqref{eq2} for $t\in[\delta,1-\delta]$ yields
\begin{align*}
\limsup_{N\rightarrow\infty} d_{\text{TV}}(W_{N},g) \leq \frac{8\delta\xi(1)}{\nu}
\end{align*}
and, letting $\delta\downarrow 0$ finishes the proof.
\qed

\end{document}